\documentclass[11pt]{amsart}

\usepackage{authblk}
\usepackage{stmaryrd}
\usepackage{amsmath}
\usepackage{amsfonts}
\usepackage{amssymb}
\usepackage{amsthm}
\usepackage{mathabx}
\usepackage[all,2cell]{xy}
\usepackage{color}
\usepackage{etoolbox}
\usepackage{tikz-cd}
\usepackage{relsize}
\usepackage{enumitem}
\usepackage{IEEEtrantools}
\usepackage{xifthen}

\UseTwocells

\newtheorem{lemma}{Lemma}
\newtheorem{proposition}{Proposition}

\newtheorem{theorem}{Theorem}

\theoremstyle{definition}

\newtheorem{remark}{Remark}

\newcommand{\red}{\color{red}}
\newcommand{\blue}{\color{blue}}

\newcommand{\e}[1]{{\emph{#1}}}
\newcommand{\ee}[1]{{{\bf #1}}}

\newbool{hidedetails}
\newcommand{\hide}[1]{\ifbool{hidedetails}{}{{\blue #1}}}

\newenvironment{eqn}{\begin{IEEEeqnarray*}{rCl}}{\end{IEEEeqnarray*}}
\newcommand{\n}{{\IEEEyesnumber}}
\newenvironment{eqnn}{\begin{IEEEeqnarray*}{rCl}}{\n\end{IEEEeqnarray*}}

\newcommand{\w}[1]{{\widetilde{#1}}}

\newcommand{\lp}{\left(}
\newcommand{\rp}{\right)}
\newcommand{\lf}{\left\{}
\newcommand{\rf}{\right\}}
\newcommand{\lb}{\left[}
\newcommand{\rb}{\right]}

\newcommand{\spec}[1]{{{\rm Spec}\lp{#1}\rp}}

\newcommand{\duap}[1]{{\lp{#1}\rp^\vee}}
\newcommand{\suspense}[1]{{#1[1]}}						
\newcommand{\streaf}[2][\bullet]{{{\mathcal O}^{#1}_{#2}}}

\newcommand{\pt}{{{\rm p}}}								
							
\newcommand{\ainfty}{{{\rm A}_\infty}}							

\newcommand{\nopos}{{\mathbb Z_{\leq 0}}}
\newcommand{\pos}{{\mathbb Z_{\geq 1}}}

\newcommand{\homdi}{{\delta}}

\newcommand{\Vari}{{X}}
\newcommand{\Hori}[1][*]{{R_{#1}}}
\newcommand{\Horit}[2][*]{{R_{#1}^{\,\otimes^{#2}}}}
\newcommand{\Proga}[1][]{{\Gamma_{#1}}}
\newcommand{\Homin}{{t}}
\newcommand{\Homi}[1]{{\Homin_{#1}}}
\newcommand{\Homod}[1][*]{{M_{#1}}}
\newcommand{\Homodi}[2]{{M_{#1,#2}}}
\newcommand{\Hobeg}{{a}}
\newcommand{\Hoend}{{b}}
\newcommand{\Hoendi}[1]{{b(#1)}}
\newcommand{\Hosub}[1][*]{{N_{#1}}}
\newcommand{\Hosubi}[2]{{N_{#1,#2}}}
\newcommand{\Hogen}[1][*]{{\w{N}_{#1}}}
\newcommand{\Hilpo}[1][]{{h_{#1}}}
\newcommand{\Hilva}{{\lambda}}
\newcommand{\Lia}[2][]{{{\mathfrak g}_{#1}^{#2}}}
\newcommand{\Liat}[3]{{{\mathfrak g}_{#1,#2}^{#3}}}
\newcommand{\Blia}[2][]{{\w{\mathfrak g}_{#1}^{#2}}}
\newcommand{\Bliat}[3]{{\w{\mathfrak g}_{[#1,#2]}^{#3}}}
\newcommand{\Coin}{{k}}
\newcommand{\Doin}{{l}}
\newcommand{\Hhom}[1][0]{{{\rm Hom}_{#1}}}
\newcommand{\Homini}{{s}}

\newcommand{\Lindet}[2]{{{\mathfrak d}_{[#1,#2]}}}
\newcommand{\Coa}[2][]{{{A}_{#1}^{#2}}}
\newcommand{\Coat}[3]{{{A}_{[#1,#2]}^{#3}}}

\newcommand{\Cindet}[2]{{d_{[#1,#2]}}}
\newcommand{\Boa}[2][\bullet]{{{\mathbf A}^{#1}_{#2}}}
\newcommand{\Voa}[1]{{{\mathbb M}_{#1}}}
\newcommand{\Voat}[2]{{{\mathbb M}_{#1,#2}}}
\newcommand{\Woa}[1]{{{\mathbf M}_{#1}}}
\newcommand{\Marat}[3][]{{{\mathbb W}^{#1}_{[#2,#3]}}}
\newcommand{\Narat}[3][]{{\underline{\mathbb W}^{#1}_{[#2,#3]}}}
\newcommand{\Mara}[1]{{{\mathbb W}_{#1}}}
\newcommand{\Varat}[3][]{{{\mathbb U}^{#1}_{[\Hobeg,#2],#3}}}
\newcommand{\Baby}{{\mathbb U}}
\newcommand{\GoodBaby}{{\mathbb V}}
\newcommand{\GooBaby}{{\mathbb W}}
\newcommand{\Xarat}[3][]{{\underline{\mathbb V}^{#1}_{[#2,#3]}}}

\newcommand{\Warat}[3][]{{\underline{\mathbb U}^{#1}_{[\Hobeg,#2],#3}}}
\newcommand{\Vara}[2][\Hobeg]{{{\mathbf U}_{[#1,#2]}}}
\newcommand{\MC}{{\gamma}}
\newcommand{\Harat}[2]{{{\mathcal H}_{[#1,#2]}}}
\newcommand{\Hara}[1]{{{\mathcal H}_{#1}}}
\newcommand{\Surj}[2]{{\pi^{#1}_{#2}}}
\newcommand{\Sarat}[2]{{\w{\mathcal H}_{[#1,#2]}}}
\newcommand{\Sara}[1]{{\w{\mathcal H}_{#1}}}
\newcommand{\Groupi}[2]{{{\mathbf G}_{#1,#2}}}
\newcommand{\Group}[1]{{{\mathbf G}_{#1}}}

\newcommand{\Gras}[1]{{{\mathbf{Gr}}_{#1}}}
\newcommand{\Grasi}[2]{{{\mathbf{Gr}}_{[#1,#2]}}}

\newcommand{\Emb}[1][]{{{\rm Emb}_{#1}}}
\newcommand{\AlgQ}[2]{{#1\slash #2}}
\newcommand{\GeoQ}[2]{{#1\sslash #2}}
\newcommand{\Shig}[2][\bullet]{{{\mathcal O}^{#1}_{#2}}}
\newcommand{\Hhig}[2][\bullet]{{\w{\mathcal O}^{#1}_{#2}}}
\newcommand{\Bhig}[2][\bullet]{{\overline{\mathcal O}^{#1}_{#2}}}

\newcommand{\Cotap}[3][\bullet]{{\Omega^{#1}\lp #2\rp|_{#3}}}
\newcommand{\Hogy}[1][\bullet]{{H^{#1}}}
\newcommand{\Hosh}[1][\bullet]{{{\mathcal H}^{#1}}}
\newcommand{\Gm}{{\mathbb C^\times}}

\booltrue{hidedetails}

\begin{document}

\title[Derived $Quot$-schemes as dg manifolds]{Shifted symplectic structures on derived Quot-stacks II -- Derived $Quot$-schemes as dg manifolds}
\maketitle

\author{Dennis Borisov, Ludmil Katzarkov, Artan Sheshmani}
{Dennis Borisov${}^{0}$ and Ludmil Katzarkov$^{3,4}$ and Artan Sheshmani${}^{1,2,3}$}
\address{${}^0$  Department of Mathematics and Statistics, University of Windsor, 401 Sunset Ave., Windsor Ontario, Canada}
\address{${}^1$  Beijing Institute of Mathematical Sciences and Applications, A6, Room 205 No. 544, Hefangkou
Village, Huaibei Town, Huairou District, Beijing 101408, China}
\address{${}^2$ Massachusetts Institute of Technology (MIT), IAiFi Institute, 182 Memorial Drive, Cambridge,
MA 02139, USA}
\address{${}^3$ National Research University Higher School of Economics, Russian Federation, Laboratory of Mirror Symmetry, NRU HSE, 6 Usacheva str.,Moscow, Russia, 119048}
\address{${}^4$ University of Miami, Department of Mathematics, 1365 Memorial Drive, Ungar 515, Coral Gables, FL 33146}

\date{\today}

\begin{abstract} It is proved that derived $Quot$-schemes, as defined by Ciocan-Fontanine and Kapranov, are represented by dg manifolds of finite type. This is the second part if a work aimed to analyze shifted symplectic structures on moduli spaces of coherent sheaves on Calabi--Yau manifolds. The first part related dg manifolds to derived schemes as defined by To\"en and Vezzosi.

\smallskip

\noindent{\bf MSC codes:} 14A20, 14N35, 14J35, 14F05, 55N22, 53D30

\noindent{\bf Keywords:} Moduli spaces of sheaves, $Quot$-schemes, derived schemes, differential graded schemes

\end{abstract}

\tableofcontents

\section{Introduction}

\subsection{The question}
Given a projective variety $\Vari$ and a coherent sheaf $\mathcal F$ on $\Vari$, the functor of quotients of $\mathcal F$ with a given Hilbert polynomial $\Hilpo_{\mathcal F}(\Hilva)-\Hilpo(\Hilva)$ is represented by a projective scheme -- the $Quot$-scheme. This is a classical result proved in \cite{Groth}. The question of a derived moduli functor was addressed in \cite{DerQuot}. To answer this question one has to start with a problem that allows a derived formulation. 

The idea of \cite{DerQuot} is as follows. Let $\Hori$ be the homogeneous ring corresponding to the projective variety $\Vari$, and let $\Homod$ be the graded $\Hori$-module corresponding to $\mathcal F$. Choosing $\Hobeg\geq 0$ large enough to ensure Castelnuovo--Mumford regularity, the classical problem becomes classifying all graded sub-modules $\Hosub[\geq\Hobeg]\subseteq\Homod[\geq\Hobeg]$ such that $\dim \Hosub[\Homin]=\Hilpo(\Homin)$ for each $\Homin\geq\Hobeg$. The derived moduli problem is to classify all such $\ainfty$-submodules. It was observed in \cite{DerQuot} that tangent complexes corresponding to this derived moduli problem have the expected cohomology given in terms of the $Ext$-functor.

\smallskip

A solution to this derived moduli problem was proposed in \cite{DerQuot}. To avoid infinite dimensionality one considers a finite stretch $\lb \Hobeg,\Homin\rb$ and classifies all graded $\Hori$-submodules $\Hosubi{\Hobeg}{\Homin}:=\underset{\Hobeg\leq\Homini\leq\Homin}\bigoplus\Hosub[\Homini]$ of $\Homodi{\Hobeg}{\Homin}:=\underset{\Hobeg\leq\Homini\leq\Homin}\bigoplus\Homod[\Homini]$ such that for any $\Hobeg\leq\Homini\leq\Homin$ $\dim \Hosub[\Homini]=\Hilpo(\Homini)$. The requirement on $\Hosubi{\Hobeg}{\Homin}$ to be an $\Hori$-submodule is expressed in terms of algebraic equations on the product of Grassmannians $\Grasi{\Hobeg}{\Homin}:=\underset{\Hobeg\leq\Homini\leq\Homin}\prod\Gras{}\lp\Hilpo(\Homini),\Homod[\Homini]\rp$ using the tautological bundles. 

Solving the derived moduli problem for the stretch $\lb \Hobeg,\Homin\rb$ results in a dg manifold $\Narat{\Hobeg}{\Homin}$ and an affine morphism $\Narat{\Hobeg}{\Homin}\longrightarrow\Grasi{\Hobeg}{\Homin}$. In \cite{DerQuot} it was shown that choosing $\Hoend\geq\Hobeg$ large enough, we have that $\forall \Homin\geq\Hoend$ the classical scheme in the dg manifold $\Narat{\Hobeg}{\Homin}$ is canonically isomorphic to the classical $Quot$-scheme. 

Moreover, it was shown that $\forall \Coin\geq 1$ there is $\Homi{\Coin}$ such that $\forall \Homin\geq\Homi{\Coin}$ the tangent complex of $\Narat{\Hobeg}{\Homin}$ has the correct cohomology in degrees $\leq\Coin$. This fact suggests that one should consider a homotopy limit of the sequence
	\begin{eqnn}\label{ProSeq}\ldots\longrightarrow\Narat{\Hobeg}{\Hoend+2}
		\longrightarrow\Narat{\Hobeg}{\Hoend+1}
		\longrightarrow\Narat{\Hobeg}{\Hoend}\end{eqnn}
as the derived $Quot$-scheme. Morphisms in (\ref{ProSeq}) are fibrations between dg manifolds, with the underlying morphisms in degree $0$ being projective. Therefore, one cannot compute a limit of (\ref{ProSeq}) within the category of dg manifolds. However, using the embedding of quasi-projective dg manifolds into derived stacks (e.g.\@ \cite{BKS}), one can compute a homotopy limit of (\ref{ProSeq}) in the category of derived stacks, thus obtaining a derived $Quot$-stack.

\smallskip

Here we arrive at the following question: \e{is the derived $Quot$-stack described above representable by a dg manifold of finite type?} We would like to emphasize that we ask for representability by \e{a dg manifold}, not just a derived scheme. The two notions are not equivalent (\cite{To14} \S1). 

A derived scheme is given by coherently gluing affine derived schemes, with the gluing done by weak equivalences. A dg manifold of finite type is a smooth classical quasi-projective scheme enhanced by a differential $\nopos$-graded structure sheaf. While any dg manifold defines a derived scheme, not every derived scheme is weakly equivalent to a dg manifold (loc.\@ cit.). The property of being representable by a dg manifold is much stronger than being a derived scheme. A derived scheme is representable by a dg manifold of finite type, if it admits a strict embedding into a projective space.

\subsection{The answer} The main result of this paper is that derived $Quot$-stacks described above are representable by dg manifolds of finite type. To explain the proof of this statement, i.e.\@ construction of the representing dg manifolds, we would like to start with a very simple example that makes the idea obvious.

\smallskip

Let $\Hosub[]$ be a $\mathbb C$-vector space of dimension $1$, and let $\Homod[]$ be a vector space of dimension $m>1$. The affine space
	\begin{eqn}\Baby:=\Hhom[]\lp\Hosub[],\Hosub[]\rp\times\Hhom[]\lp\Hosub[],\Homod[]\rp\end{eqn}
carries a right action by the multiplicative group $\Gm$. Explicitly $\forall c\in\Gm$ we have
	\begin{eqn}(\alpha,\beta)\longmapsto \lp c^{-1}\circ\alpha,\beta\circ c\rp,\end{eqn}
where we view $c$ acting by multiplication on $\Hosub[]$. We would like to divide $\Baby$ by this action of $\Gm$.\footnote{This is an elementary example of a classical problem (e.g.\@ \cite{Kraft} \S II.4.1) with $GL(n,\mathbb C)$ appearing instead of $\Gm$ in general.} The composition morphism
	\begin{eqn}\Hhom[]\lp\Hosub[],\Hosub[]\rp\times\Hhom[]\lp\Hosub[],\Homod[]\rp\longrightarrow \Hhom[]\lp\Hosub[],\Homod[]\rp\cong\Homod[]\end{eqn}
is $\Gm$-invariant, and gives us the \e{algebraic quotient} $\AlgQ{\Baby}{\Gm}$, which is a good quotient but not a geometric quotient (the fiber over $0\in\Homod[]$ consists of more than one orbit). To obtain a \e{geometric quotient} we need to restrict to a Zariski open part $\GoodBaby\subset\Baby$ consisting of pairs $(\alpha,\beta)$ where $\alpha\colon\Hosub[]\rightarrow\Hosub[]$ is surjective and $\beta\colon\Hosub[]\rightarrow\Homod[]$ is injective. Then we have a geometric quotient 
	\begin{eqn}\Homod[]\cong\AlgQ{\Baby}{\Gm}\supset\GeoQ{\GoodBaby}{\Gm}\cong\Homod[]\setminus\lf 0\rf.\end{eqn}
There is another quotient we can take. First we restrict to the injective part $\Hhom[]\lp\Hosub[],\Homod[]\rp\setminus\lf 0\rf$, take its projective quotient obtaining $\mathbb P^{m-1}$. Then we notice that $\Hhom[]\lp\Hosub[],\Hosub[]\rp$ defines a trivial line bundle on $\Hhom[]\lp\Hosub[],\Homod[]\rp\setminus\lf 0\rf$, that is (non-trivially) linearized with respect to the action by $\Gm$, and it descends to a line bundle on $\mathbb P^{m-1}$ (namely $\mathcal O(-1)$). We denote by $\GooBaby$ the total space of this $\mathcal O(-1)$, and call it a \e{partial quotient} of $\Baby$ by $\Gm$.

One immediately observes that the total space of $\mathcal O(-1)$ without the zero section is exactly $\GeoQ{\GoodBaby}{\Gm}$. Thus we have three different quotients and two open inclusions
	\begin{eqn}\AlgQ{\Baby}{\Gm}\hookleftarrow\GeoQ{\GoodBaby}{\Gm}\hookrightarrow\GooBaby.\end{eqn}
Now imagine that each one of these varieties is the degree $0$ part of a dg scheme, and each of the two open inclusions extends to an open inclusion of the dg schemes. If it so happens that classical points in $\AlgQ{\Baby}{\Gm}$ and $\GooBaby$ are all in the images of these two open inclusions, we have three weakly equivalent dg schemes. Notice that $\GooBaby$ is projective, $\GeoQ{\GoodBaby}{\Gm}$ is quasi-affine, while $\AlgQ{\Baby}{\Gm}$ is affine.

\medskip

The problem with taking a limit of (\ref{ProSeq}) is that all morphisms are projective. This is due to us starting with the Grassmannians (i.e.\@ with the partial quotients as in the example above) and then constructing the dg structure using tautological bundles (i.e.\@ descended bundles as in the example above). 

Instead, we can start with affine deformation problems (i.e.\@ dg affine spaces obtained by Koszul duality from dg Lie algebras), where we (initially) disregard all the symmetries and do not require $\Hosub\rightarrow\Homod$ to be injective. Then we notice that if we impose the condition of injectivity of $\Hosub\rightarrow\Homod$ only in the stretch $\lb\Hobeg,\Hoend\rb$, the Maurer--Cartan equation will impose this injectivity also in degrees $>\Hoend$. Thus, working over $\Grasi{\Hobeg}{\Hoend}$, we can take affine quotients instead of the partial ones, obtaining a sequence of \e{affine} morphisms between dg schemes that is weakly equivalent to (\ref{ProSeq}). Then we can compute the limit within the category of dg manifolds.

\subsection*{Contents of the paper}

In Section \ref{Oneone} we start with constructing a sequence of dg Lie algebras that parametrize $\ainfty$-structures on the pairs $\Hosubi{\Hobeg}{\Homin}$, $\Homodi{\Hobeg}{\Homin}$ for all $\Homin\geq\Hobeg$. For each $\Homin$ the dg Lie algebra is finite dimensional and, using Koszul and linear duality, it defines an affine dg manifold of finite type. Thus we obtain a sequence of fibrations of affine dg manifolds as $\Homin\rightarrow\infty$.

In Section \ref{Onetwo} we consider the conditions ensuring that $\Hosubi{\Hobeg}{\Homin}\rightarrow\Homodi{\Hobeg}{\Homin}$ is injective and prove the key fact (Proposition \ref{EqBefore}) that imposing injectivity in a long enough stretch $\lb\Hobeg,\Hoend\rb$ implies injectivity everywhere for solutions of the Maurer--Cartan equations.

\smallskip

In Section \ref{Twoone} we take partial quotients, i.e.\@ transfer all the dg structures to Grassmannians and show (Proposition \ref{Comparison}) that we recover the constructions form \cite{DerQuot}.

In Section \ref{Twotwo} we take the algebraic quotients and describe them using results of the classical invariant theory.

In Section \ref{Twothree} we prove two theorems that show geometric quotients being open dg subschemes in both the algebraic and the partial quotients, with both inclusions being weak equivalences.

\smallskip

Finally, in Section \ref{Three} we use the algebraic quotients to compute within the category of dg manifolds the homotopy limit when $\Homin\rightarrow\infty$. At first we do this in terms of dg manifolds of infinite type, just because it works better with functoriality. Then, using the fact that tangent complex of the limit has coherent cohomology, we construct a dg manifold of finite type that represents the homotopy limit.

\subsection*{Terminology and notation}\begin{itemize}
\item For a vector space $V$ we write $\duap{V}$ to mean the linear dual of $V$. Similar notation is used for vector bundles.
\item For a dg complex $\Coa{\bullet}$ we write $\Coa{*}$ to mean the underlying graded object.
\item A dg algebra $\Coa{\bullet}$ is almost free, if the underlying graded algebra $\Coa{*}$ is free. 
\item A morphism of dg algebras $\Coa[1]{\bullet}\rightarrow\Coa[2]{\bullet}$ is almost free, if $\Coa[1]{*}\rightarrow\Coa[2]{*}$ is isomorphic to the canonical inclusion  $\Coa[1]{*}\rightarrow\Coa[1]{*}\coprod\Coa[3]{*}$, where $\Coa[3]{*}$ is a free algebra.
\item Our terminology for dg schemes and dg manifolds follows \cite{BKS}.
\item In particular: a dg scheme $\lp \mathcal M,\mathcal O^\bullet_\mathcal M\rp$ is a dg manifold of finite type, if $\mathcal M$ is a smooth quasi-projective scheme and $\mathcal O^\bullet_\mathcal M$ is a sheaf of dg $\mathcal O^0_\mathcal M$-algebras, that is locally almost free over $\mathcal O^0_\mathcal M$, generated by coherent $\mathcal O^0_\mathcal M$-modules (one in each negative degree). 

If $\mathcal M$ is not of finite type or the generating sheaves for $\mathcal O^\bullet_\mathcal M$ are only quasi-coherent, we speak of dg manifolds of infinite type.
\item A classical sub-scheme of a dg scheme $\lp \mathcal M,\mathcal O^\bullet_\mathcal M\rp$ is the closed subscheme of $\mathcal M$ defined by the sheaf of ideals $\homdi\lp\mathcal O^{-1}_\mathcal M\rp$.
\item Given two vector spaces $V_1$, $V_2$ we sometimes view $V_1\oplus V_2$ as an affine space and keep writing $V_1\oplus V_2$ instead of $V_1\times V_2$.

\end{itemize}

\subsection*{Acknowledgments} \includegraphics[scale=0.2]{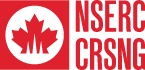} The first author acknowledges the support of the Natural Sciences and Engineering Research Council of Canada (NSERC), [RGPIN-2020-04845].

Cette recherche a été financée par le Conseil de recherches en sciences naturelles et en génie du Canada (CRSNG), [RGPIN-2020-04845].

The first author would like to thank Mikhail Kapranov for very helpful remarks and the readiness to answer questions.

\smallskip

The second author was supported by the National Science Fund of Bulgaria, National Scientific Program “VIHREN”, Project no. KP-06-DV-7, Institute of Mathematics and Informatics, Bulgarian Academy of Sciences, 1113 Sofia, Bulgaria.

\smallskip

The third author would like to thank Beijing Institute of Mathematical Sciences and Applications for excellent working conditions. {A.S. would like to further sincerely thank the Center for Mathematical Sciences and Applications at Harvard University in which the original conversations about this project were initiated and preceding articles were written, as well as the Laboratory of Mirror Symmetry in Higher School of Economics, Russian federation, for the great help and support.}

\hide{The third author would like to thank Dennis Gaitsgory, Amin Gholampour, Martijn Kool, Naichung Conan Leung, Ludmil Katzarkov and Tony Pantev for many helpful discussions and commenting on the first versions of this article. }

\section{Not taking symmetries into account}

\subsection{Classifying all $\ainfty$-structures}\label{Oneone}
Let $\lp \Vari,\streaf[]{\Vari}(1)\rp$ be a smooth, connected projective variety over $\mathbb C$, and let 
	\begin{eqn}\Hori:=\underset{\Homin\in\mathbb Z}\bigoplus\,\Hori[\Homin],\quad \Hori[\Homin]:=\Proga[]\lp\Vari,\streaf[]{\Vari}\lp\Homin\rp\rp\end{eqn} 
be its $\mathbb Z$-graded homogeneous ring. We will call $\Homin$ \ee{the degree of homogeneity}, and write it as a subscript. 

Given $\Hilpo(\Hilva)\in\mathbb Q\lb\Hilva\rb$ and a finitely generated $\mathbb Z$-graded $\Hori$-module $\Homod=\underset{\Homin\in\mathbb Z}\bigoplus\,\Homod[\Homin]$, we would like to classify all $\Hori$-submodules $\Hosub\subseteq\Homod$ that have $\Hilpo(\Hilva)$ as their Hilbert polynomial.\footnote{We will assume that $\dim_\mathbb C\Homod[\Homin]>\Hilpo(\Homin)$ for $\Homin\gg 0$.} We choose $\Hobeg\in\mathbb N$ such that each such submodule is $\Hobeg$-regular (\cite{Kleiman} Thm.\@ 1.13 p.\@ 623), in particular $\forall\Homin\geq\Hobeg$ $\dim_\mathbb C\Hosub[\Homin]=\Hilpo(\Homin)$ and the multiplication
	\begin{eqn}\Hori[\Homin-\Hobeg]\otimes\Hosub[\Hobeg]\longrightarrow\Hosub[\Homin]\end{eqn}
is surjective. This leads us to classifying $\Hosub[\geq\Hobeg]\subseteq\Homod[\geq\Hobeg]$, where $\geq\Hobeg$ means $\Homin<\Hobeg\Rightarrow\Hosub[\Homin]=\Homod[\Homin]\cong\lf 0\rf$. We do this in a derived way and, following \cite{DerQuot}, we use $\ainfty$-structures. 

\smallskip

Let $\Hosub[\geq\Hobeg]$ be a $\mathbb Z$-graded $\mathbb C$-vector space with $\dim_\mathbb C\Hosub[\Homin]=\Hilpo(\Homin)$ for $\Homin\geq\Hobeg$ and $\dim_\mathbb C\Hosub[\Homin]=0$ otherwise. Let $\Hori[+]:=\underset{\Homin\geq 1}\bigoplus\,\Hori[\Homin]$ be the irrelevant maximal ideal in $\Hori$. We view $\Hori[+]$ as a non-unital $\mathbb Z$-graded associative algebra, and first we would like to classify all $\ainfty$-module structures on $\Hosub[\geq\Hobeg]$ over $\Hori[+]$ and simultaneously all $\ainfty$-morphisms of $\ainfty$-modules $\Hosub[\geq\Hobeg]\rightarrow\Homod[\geq\Hobeg]$ over $\Hori[+]$, where $\Homod[\geq\Hobeg]$ has the given $\Hori[+]$-module structure (\cite{DerQuot} \S3.4).  For any $\Homin\geq\Hobeg$, $\Coin\geq 1$ we define
	\begin{eqn}\Lia[\Homin]{\Coin} := \Hhom\lp \Horit[+]{\Coin}\otimes\Hosub[\geq\Hobeg],\Hosub[\Homin]\rp \oplus \Hhom\lp \Horit[+]{\Coin-1}\otimes\Hosub[\geq\Hobeg],\Homod[\Homin]\rp,\end{eqn}
where $\Hhom$ stands for the space of $\mathbb C$-linear maps that have degree of homogeneity $0$.\footnote{Both $\Hosub[\Homin]$ and $\Homod[\Homin]$ are considered as $\mathbb Z$-graded vector spaces concentrated in degree $\Homin$.} Notice that $\dim_\mathbb C \Lia[\Homin]{\Coin}<\infty$. For any $\Homin_1$, $\Homin_2$, $\Coin_1$, $\Coin_2$ we have the composition map
	\begin{eqnn}\label{Compo}\Lia[\Homin_2]{\Coin_2}\otimes\Lia[\Homin_1]{\Coin_1}\longrightarrow\Lia[\Homin_2]{\Coin_1+\Coin_2},\end{eqnn}
which is necessarily trivial, if $\Homin_1\geq \Homin_2$. Alternating (\ref{Compo}) we obtain the structure of a $\pos$-graded Lie algebra on 
	\begin{eqn}\Liat{\Hobeg}{\Homin}{*}:=\underset{\underset{\Coin\geq 1}{\Hobeg\leq\Homini\leq\Homin}}{\bigoplus}\,\Lia[\Homini]{\Coin},\end{eqn}
with the grading given by $\Coin$. We will call this \ee{the homological degree} and write as a superscript. Notice that $\dim_\mathbb C \Liat{\Hobeg}{\Homin}{*}<\infty$, i.e.\@ each $\Liat{\Hobeg}{\Homin}{\Coin}$ is finite dimensional, and $\Liat{\Hobeg}{\Homin}{\Coin}\cong\lf 0\rf$ for all but finitely many $\Coin$'s. To define differentials on $\Liat{\Hobeg}{\Homin}{*}$'s we consider another sequence of finite dimensional graded Lie algebras $\Bliat{\Hobeg}{\Homin}{*}:=\underset{\underset{\Coin\geq 1}{\Hobeg\leq\Homini\leq\Homin}}{\bigoplus}\,\Blia[\Homini]{\Coin}$, where
	\begin{eqn}	\Blia[\Homini]{\Coin}:=\Lia[\Homini]{\Coin}\oplus\Hhom\lp\Horit[+]{\Coin+1},\Hori[\Homini]\rp 
	\oplus\Hhom\lp\Horit[+]{\Coin}\otimes\Homod[\geq\Hobeg],\Homod[\Homini]\rp.\end{eqn}
Multiplication on $\Hori[+]$ and the $\Hori[+]$-module structure on $\Homod[\geq\Hobeg]/\Homod[>\Homin]$ give a Maurer--Cartan element in $\Bliat{\Hobeg}{\Homin}{*}$, hence equipping $\Bliat{\Hobeg}{\Homin}{*}$ with a differential $\Lindet{\Hobeg}{\Homin}$. Clearly $\Liat{\Hobeg}{\Homin}{\bullet}:=\lp\Liat{\Hobeg}{\Homin}{*},\Lindet{\Hobeg}{\Homin}\rp$ is a dg Lie subalgebra of $\Bliat{\Hobeg}{\Homin}{\bullet}:=\lp\Bliat{\Hobeg}{\Homin}{*},\Lindet{\Hobeg}{\Homin}\rp$. We have constructed a sequence of surjective morphisms between finite dimensional differential $\pos$-graded Lie algebras
	\begin{eqnn}\label{LieSequence}\ldots\longrightarrow\Liat{\Hobeg}{\Hobeg+2}{\bullet}\longrightarrow\Liat{\Hobeg}{\Hobeg+1}{\bullet}\longrightarrow\Lia[\Hobeg]{\bullet},\end{eqnn}
where the morphisms are given by projections on direct summands. Finite dimensionality allows us to re-write (\ref{LieSequence}) in terms of finitely generated dg commutative algebras:
	\begin{eqnn}\label{ComSequence}\ldots\longleftarrow\Coat{\Hobeg}{\Hobeg+2}{\bullet}\longleftarrow\Coat{\Hobeg}{\Hobeg+1}{\bullet}\longleftarrow\Coa[\Hobeg]{\bullet},\end{eqnn}
where $\Coat{\Hobeg}{\Homin}{\bullet}=\lp\Coat{\Hobeg}{\Homin}{*},\Cindet{\Hobeg}{\Homin}\rp$, with $\Coat{\Hobeg}{\Homin}{*}$ being the free $\nopos$-graded unital commutative algebra generated by $\duap{\suspense{\Liat{\Hobeg}{\Homin}{*}}}$ and $\Cindet{\Hobeg}{\Homin}$ being the Koszul dual of $\Lindet{\Hobeg}{\Homin}$ and the Lie bracket (\cite{GinzKap} \S4). 

Let $\Boa{\Hobeg}$ be the colimit of (\ref{ComSequence}) computed in the category of dg commutative algebras. Since all algebras and all morphisms in (\ref{ComSequence}) are almost free, it is clear that $\Boa{\Hobeg}$ is an almost free dg commutative algebra. For each $\Homin$ let 
	\begin{eqn}\Voat{\Hobeg}{\Homin}:=\spec{\Coat{\Hobeg}{\Homin}{\bullet}}\end{eqn}
be the affine dg manifold of finite type defined by $\Coat{\Hobeg}{\Homin}{\bullet}$. Then (\ref{ComSequence}) becomes a sequence of fibrations of affine dg manifolds:
	\begin{eqnn}\label{ManSequence}\ldots\longrightarrow\Voat{\Hobeg}{\Hobeg+2}\longrightarrow\Voat{\Hobeg}{\Hobeg+1}\longrightarrow\Voa{\Hobeg}.\end{eqnn}
For any $\Homin\geq\Homini$ we will denote the composite map by $\Surj{\Homin}{\Homini}\colon\Voat{\Hobeg}{\Homin}\rightarrow\Voat{\Hobeg}{\Homini}$. The limit of (\ref{ManSequence}), computed in the category of stacks, is a dg affine manifold of infinite type $\Woa{\Hobeg}=\spec{\Boa{\Hobeg}}$. 

From the construction it is clear that $\Woa{\Hobeg}$ classifies all $\ainfty$-module structures on $\Hosub[\geq\Hobeg]$ and all $\ainfty$-morphisms $\Hosub[\geq\Hobeg]\rightarrow\Homod[\geq\Hobeg]$, without taking into account symmetries of $\Hosub[\geq\Hobeg]$ given by the actions of general linear groups. It is also clear that degree $0$ parts of the $\ainfty$-morphisms that $\Woa{\Hobeg}$ classifies can be arbitrary, in particular they do not have to be injective.

\subsection{Classifying $\ainfty$-submodules}\label{Onetwo}
There are two reasons for $\Woa{\Hobeg}$ not to be a suitable dg manifold to parametrize $\Hori[+]$-submodules $\Hosub[\geq\Hobeg]\subseteq\Homod[\geq\Hobeg]$:\begin{enumerate}
	\item Not all classical points in $\Woa{\Hobeg}$ correspond to submodules. Indeed, $\Woa{\Hobeg}$ parametrizes all morphisms $\Hosub[\geq\Hobeg]\rightarrow\Homod[\geq\Hobeg]$ including those that are not injective. 
	\item The internal symmetries of $\Hosub[\geq\Hobeg]$ (given by the actions of general linear groups) are not taken into account.	
\end{enumerate}
In this section we take care of the first issue. To do this we need to impose the open condition of maximality of rank in $\Hhom[\mathbb C]\lp\Hosub[\Homin],\Homod[\Homin]\rp\subseteq\Lia[\Homin]{1}$ for each $\Homin\geq\Hobeg$. This gives us open dg submanifolds $\Marat{\Hobeg}{\Homin}\subseteq\Voat{\Hobeg}{\Homin}$, $\Homin\geq\Hobeg$, and then a sequence of fibrations between quasi-affine dg manifolds
	\begin{eqnn}\label{USequence}\ldots\longrightarrow\Marat{\Hobeg}{\Hobeg+2}\longrightarrow\Marat{\Hobeg}{\Hobeg+1}\longrightarrow\Mara{\Hobeg}.\end{eqnn}
The morphisms in (\ref{USequence}) are quasi-affine, not necessarily affine, hence we cannot expect to be able to compute the limit of this sequence in the category of dg manifolds. Instead we will construct a weakly equivalent sequence of dg manifolds, almost all of whose morphisms are affine.

\smallskip

We begin with recalling, in a slightly modified form, an argument from \cite{DerQuot} \S1.4. For each $\Homin\geq\Hobeg$ we denote by $\Harat{\Hobeg}{\Homin}\subseteq\Marat{\Hobeg}{\Homin}$ the classical subscheme. This gives us a sequence of quasi-affine morphisms of schemes
	\begin{eqn}\ldots\overset{\Surj{\Hobeg+3}{\Hobeg+2}}\longrightarrow\Harat{\Hobeg}{\Hobeg+2}
	\overset{\Surj{\Hobeg+2}{\Hobeg+1}}\longrightarrow\Harat{\Hobeg}{\Hobeg+1}
	\overset{\Surj{\Hobeg+1}{\Hobeg}}\longrightarrow\Hara{\Hobeg}.\end{eqn}
For any $\Homin\geq\Homini\geq\Hobeg$ we denote by $\Surj{\Homin}{\Homini}\lp\Harat{\Hobeg}{\Homin}\rp\subseteq\Harat{\Hobeg}{\Homini}$ the scheme-theoretic image of $\Surj{\Homin}{\Homini}$. We define
	\begin{eqn}\Sarat{\Hobeg}{\Homini}:=\underset{\Homin\geq\Homini}\bigcap\,\Surj{\Homin}{\Homini}\lp\Harat{\Hobeg}{\Homin}\rp\subseteq\Harat{\Hobeg}{\Homini}.\end{eqn}
Since $\Harat{\Hobeg}{\Homini}$	is a Noetherian scheme, there is $\Hoendi{\Homini}\geq\Homini$ such that
	\begin{eqn}\Sarat{\Hobeg}{\Homini}=\Surj{\Hoendi{\Homini}}{\Homini}\lp\Harat{\Hobeg}{\Hoendi{\Homini}}\rp.\end{eqn}
In the case $\Homini=\Hobeg$ we will write $\Hoend$ instead of $\Hoendi{\Hobeg}$.

\begin{proposition}\label{GoodGen} Let $\pt\in\Sara{\Hobeg}$ be a closed point, let $\Hosub[\Hobeg]\subseteq\Homod[\Hobeg]$ be the corresponding $\mathbb C$-linear subspace, and let $\Hogen[\geq\Hobeg]\subseteq\Homod[\geq\Hobeg]$ be the $\Hori[+]$-submodule generated by $\Hosub[\Hobeg]$. Then $\forall \Homin\geq\Hobeg$ $\dim_\mathbb C\Hogen[\Homin]=\Hilpo(\Homin)$.\end{proposition}
\begin{proof} First suppose that $\pt\in\Sara{\Hobeg}\subseteq\Mara{\Hobeg}$ lifts to 
	\begin{eqnn}\label{LiftSeq}\spec{\mathbb C}\longrightarrow\lf \ldots\rightarrow \Marat{\Hobeg}{\Hobeg+2}\rightarrow\Marat{\Hobeg}{\Hobeg+1}\rightarrow\Mara{\Hobeg}\rf,\end{eqnn}
i.e.\@ we can choose $\lf\pt_\Homin\in\Marat{\Hobeg}{\Homin}\rf_{\Homin\geq\Hobeg}$ such that $\pt_\Hobeg=\pt$ and for any $\Homin\geq\Homini\geq\Hobeg$ $\Surj{\Homin}{\Homini}\lp\pt_\Homin\rp=\pt_\Homini$. This gives us an $\Hori[+]$-submodule $\Hosub[\geq\Hobeg]\subseteq\Homod[\geq\Hobeg]$ such that $\forall \Homin\geq\Hobeg$ $\dim_\mathbb C\Hosub[\Homin]=\Hilpo(\Homin)$. By assumption each such submodule is $\Hobeg$-regular, in particular it is generated by $\Hosub[\Hobeg]$, i.e.\@ $\Hosub[\geq\Hobeg]=\Hogen[\geq\Hobeg]$, and we are done.

\smallskip

We claim that every $\pt\in\Sara{\Hobeg}$ lifts to (\ref{LiftSeq}). We observe that
	\begin{eqn}\forall \Homini\geq\Hobeg\quad \Sarat{\Hobeg}{\Homini}=\underset{\Homin>\Homini}\bigcap\,\Surj{\Homin}{\Homini}\lp\Sarat{\Hobeg}{\Homin}\rp.\end{eqn}
Hence $\forall\Homini\geq\Hobeg$ any $\spec{\mathbb C}\rightarrow\Sarat{\Hobeg}{\Homini}$ lifts to $\Sarat{\Hobeg}{\Homini+1}$, implying our claim.\end{proof}

\smallskip

For any $\Homin\geq\Hoend$ let $\Varat{\Hoend}{\Homin}\subseteq\Voat{\Hobeg}{\Homin}$ be the pullback
	\begin{eqn}\xymatrix{\Varat{\Hoend}{\Homin}\ar[rr]\ar[d] && \Voat{\Hobeg}{\Homin}\ar[d]\\
											\Marat{\Hobeg}{\Hoend}\ar[rr] && \Voat{\Hobeg}{\Hoend}}\end{eqn}
computed in the category of dg manifolds. We have a sequence of affine morphisms
	\begin{eqnn}\label{QuasiAff}\ldots\longrightarrow\Varat{\Hoend}{\Hoend+2}\longrightarrow\Varat{\Hoend}{\Hoend+1}\longrightarrow\Varat{\Hoend}{\Hoend}\cong
	\Marat{\Hobeg}{\Hoend}.\end{eqnn}
	
\begin{proposition}\label{EqBefore} For any $\Homin\geq\Hoend$ the inclusion $\Marat{\Hobeg}{\Homin}\hookrightarrow\Varat{\Hoend}{\Homin}$ is a weak equivalence.\end{proposition}
\begin{proof} By construction $\Marat{\Hobeg}{\Homin}\hookrightarrow\Varat{\Hoend}{\Homin}$ is an open inclusion, hence it is enough to show that the morphism between the corresponding classical schemes is an isomorphism. This classical morphism is also an open inclusion, hence it is enough to show that the corresponding map between sets of closed points is surjective.

A closed point $\pt\in\Varat{\Hoend}{\Homin}$ is given by a Maurer--Cartan element $\MC$ in $\Liat{\Hobeg}{\Homin}{\bullet}$, which defines an $\Hori[+]$-module structure on $\Hosubi{\Hobeg}{\Homin}$ and a morphism of $\Hori[+]$-modules $\Hosubi{\Hobeg}{\Homin}\rightarrow\Homodi{\Hobeg}{\Homin}$. Since $\Homin\geq\Hoend$ we have $\Surj{\Homin}{\Hobeg}\lp\pt\rp\in\Sara{\Hobeg}$. Let $\Hogen[\geq\Hobeg]\subseteq\Homodi{\Hobeg}{\Homin}$ be the $\Hori[+]$-submodule generated by $\Hosub[\Hobeg]\subseteq\Homod[\Hobeg]$. According to Proposition \ref{GoodGen} we have $\forall \Homini\geq\Hobeg$ $\dim_\mathbb C\Hogen[\Homini]=\Hilpo(\Homini)$. For any $\Homin\geq\Homini>\Hoend$ the image of $\Hosub[\Homini]\rightarrow\Homod[\Homini]$ has dimension $\leq\Hilpo(\Homini)$ and it must contain $\Hogen[\Homini]$. Thus $\Hosubi{\Hobeg}{\Homin}\rightarrow\Homodi{\Hobeg}{\Homin}$ is injective, i.e.\@ $\pt\in\Marat{\Hobeg}{\Homin}$.\end{proof}

\hide{\smallskip

Since every morphism in (\ref{QuasiAff}) is affine we can compute the limit of this sequence in the category of dg manifolds (of infinite type). We denote this limit by  $\Vara{\Hoend}$. Choosing any $\Hobeg'\geq\Hobeg$ we obtain the corresponding $\Hoend':=\Hoendi{\Hobeg'}$. Let $\Hoend'':=\max (\Hoend,\Hoend')$. We {\red have} canonical weak equivalences
	\begin{eqn}\Vara[\Hobeg]{\Hoend}\longleftarrow\Vara[\Hobeg]{\Hoend''}\longrightarrow\Vara[\Hobeg']{\Hoend''}\longrightarrow\Vara[\Hobeg']{\Hoend'}.\end{eqn}}

\section{Dividing by the symmetries}
We fix the natural numbers $\Hobeg\leq\Hoend$ as in the previous section. For any $\Homini\geq\Hobeg$ let $\Group{\Homini}:={\rm GL}\lp\Hilpo(\Homini),\mathbb C\rp$. Choosing a basis in $\Hosub[\Homini]$ we obtain a left action of $\Group{\Homini}$ on $\Hosub[\Homini]$ (recall that $\dim\Hosub[\Homini]=\Hilpo(\Homini)$). For each $\Coin\geq 1$ and each $\Homin\geq\Homini$ this gives us a right action of $\Group{\Homini}$ on 
	\begin{eqn}\Hhom\lp\Horit[+]{\Coin}\otimes\Hosub[\Homini],\Hosubi{\Homini+1}{\Homin}\rp \oplus
	\Hhom\lp\Horit[+]{\Coin-1}\otimes\Hosub[\Homini],\Homodi{\Homini}{\Homin}\rp,\end{eqn}
and a left action on $\Hhom\lp\Horit[+]{\Coin}\otimes\Hosubi{\Hobeg}{\Homini-1},\Hosub[\Homini]\rp$. Using inverses in $\Group{\Homini}$ we view the latter as a right action as well. Composition of morphisms is clearly invariant with respect to this action. Altogether we have a right action of $\Group{\Homini}$ on the dg affine manifold $\Voat{\Hobeg}{\Homin}$.  For any $\Homin\geq\Hoend$ we denote 
	\begin{eqn}\Groupi{\Hobeg}{\Homin}:=\underset{\Hobeg\leq\Homini\leq\Homin}\prod\Group{\Homini}.\end{eqn}
It is clear that the actions of $\Group{\Homini}$, $\Group{\Homini'}$ on $\Voat{\Hobeg}{\Homin}$ commute, if $\Homini\neq\Homini'$. Hence $\Voat{\Hobeg}{\Homin}$ carries a right action of $\Groupi{\Hobeg}{\Homin}$. It is also clear that the open dg submanifolds $\Marat{\Hobeg}{\Homin}\subseteq\Varat{\Hoend}{\Homin}\subseteq\Voat{\Hobeg}{\Homin}$ are invariant with respect to this action, since they are defined by putting conditions on ranks of the homomorphisms.

We would like to divide by the action of $\Groupi{\Hobeg}{\Homin}$ and obtain (in the limit $\Homin\rightarrow\infty$) the dg Quot-scheme as a result. Here we have to be careful, as there are different ways of taking quotients. There are not only different conditions one can put on the fibers of a quotient map, but also constructions themselves can be different. We will consider $3$ of them.

\subsection{Taking a partial quotient}\label{Twoone}
The dg manifold $\Marat{\Hobeg}{\Homin}$ is quasi-affine, not affine, hence we should not expect a meaningful quotient obtained by taking the $\Groupi{\Hobeg}{\Homin}$-invariant globally defined functions. Instead, a projective quotient should be constructed. Let 
	\begin{eqn}\Emb[{[\Hobeg,\Homin]}]:=\underset{\Hobeg\leq\Homini\leq\Homin}\bigoplus\Emb[\Homini]\lp \Hosub[\Homini],\Homod[\Homini]\rp\subset
	\underset{\Hobeg\leq\Homini\leq\Homin}\bigoplus\Hhom[\mathbb C]\lp \Hosub[\Homini],\Homod[\Homini]\rp\end{eqn}
be the open subset consisting of all injective morphisms. Denoting by $\Marat[0]{\Hobeg}{\Homin}$ the degree $0$ part of the dg manifold $\Marat{\Hobeg}{\Homin}$ we have an affine morphism
	\begin{eqnn}\label{LongBundle}\Marat[0]{\Hobeg}{\Homin}\longrightarrow\Emb[{[\Hobeg,\Homin]}].\end{eqnn}
In fact, it is a trivial vector bundle, whose fiber is $\Hhom\lp\Hori[+]\otimes\Hosubi{\Hobeg}{\Homin},\Hosubi{\Hobeg}{\Homin}\rp$. Taking a projective quotient of $\Emb[{[\Hobeg,\Homin]}]$ by the action of $\Groupi{\Hobeg}{\Homin}$ we obtain a product of Grassmannians
	\begin{eqn}\Grasi{\Hobeg}{\Homin}:=\underset{\Hobeg\leq\Homini\leq\Homin}\prod \Gras{}\lp\Hilpo(\Homini),\dim\Homod[\Homini]\rp.\end{eqn}
\hide{(e.g.\@ \cite{Mukai} Remark 6.14 (iv), p.\@ 194).} The quotient map $\Emb[{[\Hobeg,\Homin]}]\longrightarrow\Grasi{\Hobeg}{\Homin}$ is a principal $\Groupi{\Hobeg}{\Homin}$-bundle, hence $\Groupi{\Hobeg}{\Homin}$-linearized coherent sheaves on $\Emb[{[\Hobeg,\Homin]}]$ descend functorially to $\Grasi{\Hobeg}{\Homin}$ (e.g.\@ \cite{HuyLehn} Thm.\@ 4.2.14 p.\@ 98). E.g.\@ (\ref{LongBundle}) descends to a vector bundle on $\Grasi{\Hobeg}{\Homin}$. 

The descent construction is given by first taking the direct image and then the $\Groupi{\Hobeg}{\Homin}$-invariant sections. It is functorial with respect to invariant morphisms, therefore the dg structure sheaf of $\Marat{\Hobeg}{\Homin}$ descends to a dg structure sheaf, giving us a dg manifold $\Narat{\Hobeg}{\Homin}$ with an affine morphism $\Narat{\Hobeg}{\Homin}\rightarrow\Grasi{\Hobeg}{\Homin}$.
	
\begin{proposition}\label{Comparison} For any $\Homin\geq\Hobeg$ the dg manifold $\Narat{\Hobeg}{\Homin}$ is isomorphic to the derived Quot-scheme $RG_A(h,M_{[\Hobeg,\Homin]})$ defined in \cite{DerQuot}, page 435. \end{proposition}
\begin{proof} The only difference between $\Narat{\Hobeg}{\Homin}$ and the construction in \cite{DerQuot} is in the order of the following two operations: defining the dg Lie algebra of multi-linear maps that encodes $\ainfty$-structures on the pair $\Hosubi{\Hobeg}{\Homin}$, $\Homodi{\Hobeg}{\Homin}$, and taking the quotient with respect to the action by $\Groupi{\Hobeg}{\Homin}$.

In loc.\@ cit.\@ one starts with $\Grasi{\Hobeg}{\Homin}$, takes all the necessary multi-linear maps involving the trivial bundle with fiber $\Homodi{\Hobeg}{\Homin}$ and its tautological sub-bundle on $\Grasi{\Hobeg}{\Homin}$. Then one imposes the condition that the linear map from the tautological sub-bundle to $\Homodi{\Hobeg}{\Homin}$ is the tautological inclusion.

Here we have started with $\underset{\Hobeg\leq\Homini\leq\Homin}\bigoplus \Emb\lp \Hosub[\Homini],\Homod[\Homini]\rp$ instead of $\Grasi{\Hobeg}{\Homin}$, took all the necessary multi-linear maps, and only then passed to $\Grasi{\Hobeg}{\Homin}$ by taking the projective quotient. The results are canonically isomorphic.\end{proof}

\begin{remark} Notice that the isomorphism in Proposition \ref{Comparison} is independent of any choices, i.e.\@ it is indeed canonical. Moreover, the morphisms in (\ref{USequence}) are equivariant with respect to the group actions, i.e.\@ with respect to the obvious projections $\lf \Groupi{\Hobeg}{\Homin}\rightarrow\Groupi{\Hobeg}{\Homin-1}\rf_{\Homin>\Hobeg}$. This implies that we obtain a sequence of fibrations of dg manifolds
	\begin{eqnn}\label{UGSequence}\ldots\longrightarrow\Narat{\Hobeg}{\Hoend+2}
	\longrightarrow\Narat{\Hobeg}{\Hoend+1}
	\longrightarrow\Narat{\Hobeg}{\Hoend}.\end{eqnn}
This tower of dg manifolds was considered in \cite{DerQuot} Theorem 4.3.2. Our goal is to show that the limit of this sequence, computed in the category of stacks, is represented by a dg manifold.\end{remark}

\subsection{Taking an algebraic quotient}\label{Twotwo}
Also $\Varat{\Hoend}{\Homin}$ is a quasi-affine dg manifold and also here we separate the affine and the quasi-affine parts by considering the affine morphism
	\begin{eqnn}\label{AlgBundle}\Varat[0]{\Hoend}{\Homin}\longrightarrow\Emb[{[\Hobeg,\Hoend]}]\end{eqnn}
which is a trivial vector bundle with the fiber
	\begin{eqn}\Hhom\lp\Hosubi{\Hoend+1}{\Homin},\Homodi{\Hoend+1}{\Homin}\rp\oplus
	\Hhom\lp\Hori[+]\otimes\Hosubi{\Hobeg}{\Homin},\Hosubi{\Hobeg}{\Homin}\rp.\end{eqn}
This bundle and the rest of the dg structure sheaf of $\Varat{\Hoend}{\Homin}$ descend to a projective quotient of $\Emb[{[\Hobeg,\Hoend]}]$ with respect to the action of $\Groupi{\Hobeg}{\Hoend}$, giving us a dg manifold fibered over $\Grasi{\Hobeg}{\Hoend}$:
	\begin{eqnn}\label{TheAffine}\Warat{\Hoend}{\Homin}\longrightarrow\Grasi{\Hobeg}{\Hoend}.\end{eqnn}
Notice that $\Grasi{\Hobeg}{\Hoend}$ is fixed, i.e.\@ it is independent of $\Homin$. It is clear that (\ref{TheAffine}) is an affine morphism and it factors through $\Narat{\Hobeg}{\Hoend}$: 
	\begin{eqnn}\label{AlgBunPro}\xymatrix{\Warat{\Hoend}{\Homin}\ar[rr]\ar[dr] && \Narat{\Hobeg}{\Hoend}\ar[dl]\\
	& \Grasi{\Hobeg}{\Hoend} &} \end{eqnn}
The group $\Groupi{\Hobeg}{\Hoend}$ acts trivially on  (\ref{AlgBunPro}). On the other hand, $\Groupi{\Hoend+1}{\Homin}$ acts trivially on $\Narat{\Hobeg}{\Hoend}\rightarrow\Grasi{\Hobeg}{\Hoend}$ but not trivially on $\Warat{\Hoend}{\Homin}$. We would like to divide by this action of $\Groupi{\Hoend+1}{\Homin}$.

\smallskip

It is here that we take an algebraic quotient. Namely, using the fact that (\ref{TheAffine}) is affine and working locally on $\Grasi{\Hobeg}{\Hoend}$, we take the $\Groupi{\Hoend+1}{\Homin}$-invariant elements of the dg structure sheaf of $\Warat{\Hoend}{\Homin}$. We denote the result by
	\begin{eqnn}\label{AlgQuo}\xymatrix{\AlgQ{\Warat{\Hoend}{\Homin}}{\Groupi{\Hoend+1}{\Homin}}\ar[rr]\ar[dr] && \Narat{\Hobeg}{\Hoend}\ar[dl]\\
		& \Grasi{\Hobeg}{\Hoend} &}\end{eqnn}
Obvious functoriality of this construction gives us a sequence of affine morphisms of dg schemes:		
	\begin{eqnn}\label{AffineSeq}\ldots\longrightarrow\AlgQ{\Warat{\Hoend}{\Hoend+3}}{\Groupi{\Hoend+1}{\Hoend+3}}\longrightarrow
	\AlgQ{\Warat{\Hoend}{\Hoend+2}}{\Groupi{\Hoend+1}{\Hoend+2}}\longrightarrow\AlgQ{\Warat{\Hoend}{\Hoend+1}}{\Group{\Hoend+1}}.\qquad\end{eqnn}

Notice that $\lf \AlgQ{\Warat{\Hoend}{\Homin}}{\Groupi{\Hoend+1}{\Homin}}\rf_{\Homin>\Hoend}$ are not dg manifolds in general. Indeed, each $\Warat[0]{\Hoend}{\Homin}\rightarrow\Grasi{\Hobeg}{\Hoend}$ is affine, hence it is clear that 
	\begin{eqnn}\label{GoodQuo}\Warat[0]{\Hoend}{\Homin}\longrightarrow\AlgQ{\Warat[0]{\Hoend}{\Homin}}{\Groupi{\Hoend+1}{\Homin}}\end{eqnn} 
is a good quotient. To describe it we work locally on $\Narat{\Hobeg}{\Hoend}$ and assume that $\Warat[0]{\Hoend}{\Homin}\rightarrow\Narat{\Hobeg}{\Hoend}$ is a trivial bundle. Since $\Groupi{\Hoend+1}{\Homin}$ acts trivially on $\Narat{\Hobeg}{\Hoend}$, it is enough to compute the quotient of the fiber of this trivial bundle, which is 
	\begin{eqnn}\label{TheFiber}\Hhom\lp\Hosubi{\Hoend+1}{\Homin},\Homodi{\Hoend+1}{\Homin}\rp\oplus
		\Hhom\lp\Hori[+]\otimes\Hosubi{\Hobeg}{\Homin},\Hosubi{\Hoend+1}{\Homin}\rp.\end{eqnn} 
Consider the following morphism, given by taking all possible compositions,
	\begin{eqnn}\label{ProjectionMap} &&\Hhom\lp\Hosubi{\Hoend+1}{\Homin},\Homodi{\Hoend+1}{\Homin}\rp\oplus
		\Hhom\lp\Hori[+]\otimes\Hosubi{\Hobeg}{\Homin},\Hosubi{\Hoend+1}{\Homin}\rp\longrightarrow\\
		&& \longrightarrow \underset{\Hobeg\leq\Homini'\leq\Hoend}{\underset{\Hoend<\Homini\leq\Homin}{\underset{1\leq \Coin\leq\Homin-\Hoend}\bigoplus}}\;
		\underset{\Homini_j>0}{{\underset{\Homini_1+\ldots+\Homini_{\Coin}=\Homini-\Homini'}\bigoplus}}
		\Hhom[\mathbb C]\lp \Hori[\Homini_1]\otimes\ldots\otimes\Hori[\Homini_\Coin]\otimes\Hosub[\Homini'],\Homod[\Homini] \rp.\end{eqnn}
It is clear that (\ref{ProjectionMap}) is $\Groupi{\Hoend+1}{\Homin}$-invariant. The image of (\ref{ProjectionMap}) consists of points whose projection to each
	\begin{eqnn}\label{ManyIma}\Hhom[\mathbb C]\lp \Hori[\Homini_1]\otimes\ldots\otimes\Hori[\Homini_\Coin]\otimes\Hosub[\Homini'],\Homod[\Homini] \rp\end{eqnn}
is of rank $\leq\Hilpo(\Homini)$. According to the classical invariant theory (e.g.\@ \cite{Kraft} \S II.3.4, \S II.4.1) this image (a reduced scheme) is a good quotient of (\ref{TheFiber}) by $\Groupi{\Hoend+1}{\Homin}$. \hide{To check that the conditions of \cite{Kraft} \S II.3.4 are satisfied, we need to check that the proposed quotient is normal, the map is surjective, and each fiber has only one closed orbit. The first two conditions are immediate from loc.\@ cit.\@ \S II.4.1. The last condition is easily seen by taking step-by-step quotient.}

\smallskip

The fibers of (\ref{ProjectionMap}) may consist of more than one orbit. To have a geometric quotient we need to put conditions of maximality of rank on elements of (\ref{ManyIma}), e.g.\@ \cite{Kraft} \S II.4.1 p.\@ 121. Outside of the locus of a geometric quotient the conditions of Thm.\@ 4.2.15 in \cite{HuyLehn} p.\@ 98 are not satisfied for the structure sheaf of $\Warat{\Hoend}{\Homin}$, i.e.\@ it does not descend to $\AlgQ{\Warat[0]{\Hoend}{\Homin}}{\Groupi{\Hoend+1}{\Homin}}$. Therefore, by taking the $\Groupi{\Hoend+1}{\Homin}$-invariant elements of the dg structure sheaf of $\Warat{\Hoend}{\Homin}$ we obtain a dg scheme, but not necessarily a dg manifold.

\subsection{Taking a geometric quotient}\label{Twothree}
Here we would like to specify an open subscheme of $\Warat[0]{\Hoend}{\Homin}$, where the quotient (\ref{GoodQuo}) is geometric. Recall that $\Marat[0]{\Hobeg}{\Homin}\subseteq\Varat[0]{\Hoend}{\Homin}$ was defined by requiring injectivity of elements in $\Hhom\lp\Hosubi{\Hobeg}{\Homin},\Homodi{\Hobeg}{\Homin}\rp$. Let $\Xarat[0]{\Hobeg}{\Homin}\subseteq\Warat[0]{\Hoend}{\Homin}$ be the open subset given by requiring this injectivity and an additional condition: for any $\Homin\geq\Homini>\Homini'\geq\Hoend$ we take only those elements in $\Hhom[\mathbb C]\lp\Hori[\Homini-\Homini']\otimes\Hosub[\Homini'],\Hosub[\Homini]\rp$ that are surjective. 

Let $\Xarat{\Hobeg}{\Homin}\subseteq\Warat{\Hoend}{\Homin}$ be the open dg submanifold obtained by restricting the dg structure sheaf of $\Warat{\Hoend}{\Homin}$ to $\Xarat[0]{\Hobeg}{\Homin}$. Being given by putting conditions on ranks of maps, $\Xarat[0]{\Hobeg}{\Homin}$ is clearly a $\Groupi{\Hoend+1}{\Homin}$-invariant open part of $\Warat[0]{\Hoend}{\Homin}$. Moreover, this is the pre-image of the open subscheme 
	\begin{eqn}\GeoQ{\Xarat[0]{\Hobeg}{\Homin}}{\Groupi{\Hoend+1}{\Homin}}\subseteq\AlgQ{\Warat[0]{\Hoend}{\Homin}}{\Groupi{\Hoend+1}{\Homin}}\end{eqn} 
given by requiring that components in each (\ref{ManyIma}) are of highest possible rank, i.e.\@ $\Hilpo(\Homini)$.

As indicated by the notation, $\Xarat[0]{\Hobeg}{\Homin}\rightarrow\GeoQ{\Xarat[0]{\Hobeg}{\Homin}}{\Groupi{\Hoend+1}{\Homin}}$ is a geometric quotient (e.g.\@ \cite{Kraft} \S II.4.1). It can be described in two ways: \begin{itemize}
	\item taking (locally on $\Grasi{\Hobeg}{\Hoend}$) the subring of invariant functions on $\Xarat[0]{\Hobeg}{\Homin}$,
	\item dividing $\Marat[0]{\Hobeg}{\Homin}$ by $\Groupi{\Hobeg}{\Homin}$ and then taking the open part in $\Narat[0]{\Hobeg}{\Homin}$ given by requiring surjectivity of the maps above.\end{itemize}
In other words there is a canonical open inclusion $\GeoQ{\Xarat[0]{\Hobeg}{\Homin}}{\Groupi{\Hoend+1}{\Homin}}\hookrightarrow\Narat[0]{\Hobeg}{\Homin}$.

\begin{theorem}\label{ToPartial} For any $\Homin\geq\Hoend$ the open inclusion $\GeoQ{\Xarat{\Hobeg}{\Homin}}{\Groupi{\Hoend+1}{\Homin}}\hookrightarrow\Narat{\Hobeg}{\Homin}$ is a weak equivalence.\end{theorem}
\begin{proof} From the proof of Proposition \ref{EqBefore} we know that any classical point in $\Marat{\Hobeg}{\Homin}$ defines a sub-module of $\Homodi{\Hobeg}{\Homin}$ that generates a sub-module of $\Homod[\geq\Hobeg]$ with $\Hilpo(\Hilva)$ as its Hilbert function. This and the fact that all such sub-modules are $\Hobeg$-regular imply that every $\Hori[\Homini-\Homini']\otimes\Hosub[\Homini']\rightarrow\Hosub[\Homini]$ is surjective for this sub-module, i.e.\@ the classical point has to lie in $\GeoQ{\Xarat{\Hobeg}{\Homin}}{\Groupi{\Hoend+1}{\Homin}}$.\end{proof}

\begin{theorem} For each $\Homin\geq\Hoend$ the open inclusion $\GeoQ{\Xarat{\Hobeg}{\Homin}}{\Groupi{\Hoend+1}{\Homin}}\subseteq\AlgQ{\Warat{\Hoend}{\Homin}}{\Groupi{\Hoend+1}{\Homin}}$ is a weak equivalence.\end{theorem}
\begin{proof} As we have noticed above the structure sheaf of $\Warat{\Hoend}{\Homin}$ does not descend to the quotient by ${\Groupi{\Hoend+1}{\Homin}}$. However, there are parts of this structure sheaf that do descend. Consider the trivial bundle on $\Warat[0]{\Hoend}{\Homin}$ having $\Hhom\lp\Hori[+]\otimes \Hosubi{\Hobeg}{\Hoend},\Homodi{\Hoend+1}{\Homin}\rp$ as the fiber. It carries a trivial action of ${\Groupi{\Hoend+1}{\Homin}}$. Therefore it does descend to the quotient (e.g.\@ Thm.\@ 4.2.15 in \cite{HuyLehn} p.\@ 98). 

The descended bundle on the quotient comes with a section, that has to vanish at classical points in $\AlgQ{\Warat{\Hoend}{\Homin}}{\Groupi{\Hoend+1}{\Homin}}$.\footnote{This section comes from the invariant section of the original bundle on $\Warat[0]{\Hoend}{\Homin}$, that is part of the structure of a dg manifold.} Vanishing of this section immediately implies that the corresponding element in $\Hhom\lp\Hosubi{\Hoend+1}{\Homin},\Homodi{\Hoend+1}{\Homin}\rp$ is injective. Then, as in the proof of Theorem \ref{ToPartial}, we see that the classical point has to lie in $\GeoQ{\Xarat{\Hobeg}{\Homin}}{\Groupi{\Hoend+1}{\Homin}}$, i.e.\@ $\GeoQ{\Xarat{\Hobeg}{\Homin}}{\Groupi{\Hoend+1}{\Homin}}$ and $\AlgQ{\Warat{\Hoend}{\Homin}}{\Groupi{\Hoend+1}{\Homin}}$ have the same classical points. Since $\GeoQ{\Xarat{\Hobeg}{\Homin}}{\Groupi{\Hoend+1}{\Homin}}\subseteq\AlgQ{\Warat{\Hoend}{\Homin}}{\Groupi{\Hoend+1}{\Homin}}$ is an open inclusion, we conclude that this is a weak equivalence.\end{proof}

\medskip

The two theorems above give us an infinite commutative diagram, whose vertical arrows are weak equivalences:

	\begin{eqn}\xymatrix{\ldots\ar[r] & \AlgQ{\Warat{\Hoend}{\Hoend+3}}{\Groupi{\Hoend+1}{\Hoend+3}}\ar[r] &
	\AlgQ{\Warat{\Hoend}{\Hoend+2}}{\Groupi{\Hoend+1}{\Hoend+2}}\ar[r] & \AlgQ{\Warat{\Hoend}{\Hoend+1}}{\Group{\Hoend+1}}\\
	\ldots\ar[r] & \GeoQ{\Xarat{\Hobeg}{\Hoend+3}}{\Groupi{\Hoend+1}{\Hoend+3}}\ar[r]\ar[d]\ar[u] &
		\GeoQ{\Xarat{\Hobeg}{\Hoend+2}}{\Groupi{\Hoend+1}{\Hoend+2}}\ar[r]\ar[d]\ar[u] & \GeoQ{\Xarat{\Hobeg}{\Hoend+1}}{\Group{\Hoend+1}}\ar[d]\ar[u]\\
	\ldots\ar[r] & \Narat{\Hobeg}{\Hoend+3}\ar[r] & \Narat{\Hobeg}{\Hoend+2}\ar[r] & \Narat{\Hobeg}{\Hoend+1}
	}\end{eqn}
Homotopy limit of the bottom row is the stack we would like to represent. It is weakly equivalent to a homotopy limit of the top row, which is a sequence of \e{affine} morphisms between dg schemes. A homotopy limit of such sequence can be computed within the category of dg manifolds. This (standard) construction is subject of the following section.

\section{Representation by a dg manifold}\label{Three}
In the previous section we have constructed a diagram
	\begin{eqnn}\label{ToCompute}\ldots\longrightarrow\AlgQ{\Warat{\Hoend}{\Hoend+2}}{\Groupi{\Hoend+1}{\Hoend+2}}\longrightarrow
	\AlgQ{\Warat{\Hoend}{\Hoend+1}}{\Group{\Hoend+1}}\end{eqnn}
of dg schemes, whose limit, computed in the category of stacks, is the derived $Quot$-scheme. In this section we would like to describe this limit using dg manifolds. To this end we recall that each $\AlgQ{\Warat{\Hoend}{\Homin}}{\Groupi{\Hoend+1}{\Homin}}$ comes with an affine morphism to $\Grasi{\Hobeg}{\Hoend}$. Therefore, instead of the limit of (\ref{ToCompute}) we can compute a colimit of the diagram
	\begin{eqnn}\label{AsSheaves}\ldots\longleftarrow\Shig{\Hoend+2}\longleftarrow\Shig{\Hoend+1}\end{eqnn}
of sheaves of differential $\nopos$-graded $\Shig[]{\Grasi{\Hobeg}{\Hoend}}$-algebras. To ensure that the colimit is homotopically correct, we need to take a resolution of (\ref{AsSheaves}).

\begin{proposition} The diagram (\ref{AsSheaves}) of differential $\nopos$-graded $\Shig[]{\Grasi{\Hobeg}{\Hoend}}$-algebras is weakly equivalent to a diagram
	\begin{eqnn}\label{GoodSheaves}\ldots\longleftarrow\Hhig{\Hoend+2}\longleftarrow\Hhig{\Hoend+1},\end{eqnn}
where each $\Hhig{\Homin}$ is a sheaf of almost free differential $\nopos$-graded $\Shig[]{\Grasi{\Hobeg}{\Hoend}}$-algebras, generated by sequences of locally free quasi-coherent sheaves on $\Grasi{\Hobeg}{\Hoend}$. Each morphism in (\ref{GoodSheaves}) is almost free.\end{proposition}
\begin{proof} As $\Grasi{\Hobeg}{\Hoend}$ is a projective scheme, every quasi-coherent sheaf is a quotient of a sum of line bundles. Then the standard construction of an almost free resolution of dg algebras, which is functorial, produces (\ref{GoodSheaves}) from (\ref{AsSheaves}), where each $\Hhig[*]{\Homin}$ is freely generated (as a commutative algebra) by a sequence of (infinite) sums of line bundles, one in each degree.\end{proof}

Let $\Hhig{\infty}$ be the categorical colimit of (\ref{GoodSheaves}). It is clear that, locally on $\Grasi{\Hobeg}{\Hoend}$, this colimit is also the homotopy colimit, giving us a Zariski atlas of a dg manifold of infinite type, that represents the homotopy limit of (\ref{UGSequence}), computed in the category of stacks.

\begin{theorem} There is a weak equivalence $\Bhig{\infty}\overset{\simeq}\longrightarrow\Hhig{\infty}$, where $\Bhig{\infty}$ is a sheaf of almost free differential $\nopos$-graded $\Shig[]{\Grasi{\Hobeg}{\Hoend}}$-algebras, generated by a sequence of locally free coherent sheaves.\end{theorem}
\begin{proof} From Proposition \ref{Comparison} and Thm.\@ 1.4.1 in \cite{DerQuot} we know that $\forall\Homin\geq\Hoend$ the classical part of $\AlgQ{\Warat{\Hoend}{\Homin}}{\Groupi{\Hoend+1}{\Homin}}$ is the classical $Quot$-scheme. In other words, the projection
	\begin{eqn}\AlgQ{\Warat{\Hoend}{\Homin+1}}{\Groupi{\Hoend+1}{\Homin+1}}\longrightarrow\AlgQ{\Warat{\Hoend}{\Homin}}{\Groupi{\Hoend+1}{\Homin}}\end{eqn}
induces an isomorphism on cohomology in degree $0$.

\begin{lemma} For any $\Coin\leq -1$ there is $\Homi{\Coin}\geq\Hoend$ such that $\forall\Homin\geq\Homi{\Coin}$
	\begin{eqnn}\label{IsoTop}\Narat{\Hobeg}{\Homin+1}\longrightarrow\Narat{\Hobeg}{\Homin}\end{eqnn}
induces isomorphisms in degrees $\geq\Coin$ of cohomologies of the structure sheaves.\end{lemma}
\begin{proof} Cohomologies of the structure sheaves of $\Narat{\Hobeg}{\Homin+1}$, $\Narat{\Hobeg}{\Homin}$ are coherent sheaves on the (common) scheme of classical points. Therefore, to prove that (\ref{IsoTop}) induces  isomorphisms on cohomologies in degrees $\geq\Coin$, it is enough work locally at each classical point.

According to \cite{DerQuot} \S4.3 for any $\Coin\leq 0$ there is $\Homi{\Coin}\geq 0$ such that $\forall\Homin\geq\Homi{\Coin}$ and for any classical point $\pt\colon\spec{\mathbb C}\rightarrow\Narat{\Hobeg}{\Homin+1}$ the map (\ref{IsoTop}) induces isomorphisms in degrees $\geq\Coin$ of cohomologies of the cotangent complexes
	\begin{eqn}\Hogy[*]\lp\Cotap{\Narat{\Hobeg}{\Homin}}{\pt}\rp\longrightarrow
	\Hogy[*]\lp\Cotap{\Narat{\Hobeg}{\Homin+1}}{\pt}\rp.\end{eqn}
\hide{Moreover, at each classical point cohomologies in degrees $\leq\Coin$ of the cotangent complex are isomorphic to the value of the $Ext$-functor. Therefore, since the classical $Quot$-scheme is Noetherian, choosing $\Coin$ large enough we obtain $\Homi{\Coin}$ such that for any $\Homin\geq\Homi{\Coin}$, $\Coin\leq\Doin\leq\Coin+2$
	\begin{eqn}\Hogy[\Doin]\lp\Cotap{\AlgQ{\Warat{\Hoend}{\Homin}}{\Groupi{\Hoend+1}{\Homin}}}{\pt}\rp\cong
	\Hogy[\Doin]\lp\Cotap{\AlgQ{\Warat{\Hoend}{\Homin+1}}{\Groupi{\Hoend+1}{\Homin+1}}}{\pt}\rp\cong\lf 0\rf.\end{eqn}}
This allows us, in a neighbourhood of each $\pt$, to find an acyclic dg ideal in the structure sheaf of $\Narat{\Hobeg}{\Homin+1}$, such that, after dividing by this ideal, (\ref{IsoTop}) becomes an isomorphism in degrees $\geq\Coin$. This isomorphism induces then an isomorphism on cohomology in degrees $\geq \Coin+1$.\end{proof}

The Lemma above immediately implies that $\forall \Coin<0$ the sheaf of cohomologies $\Hosh[\Coin]\lp\Hhig{\infty}\rp$ is a coherent sheaf on the classical $Quot$-scheme. This allows us to construct $\Bhig{\infty}$ as a subsheaf of $\Hhig{\infty}$. Indeed, we start with $\Bhig[0]{\infty}:=\Shig[]{\Grasi{\Hobeg}{\Hoend}}$. Since the ideal sheaf of the classical $Quot$-scheme in $\Grasi{\Hobeg}{\Hoend}$ is coherent, we can find a locally free coherent subsheaf of $\Hhig[-1]{\infty}$, whose differential belongs to $\Shig[]{\Grasi{\Hobeg}{\Hoend}}$ and equals the ideal sheaf of the classical $Quot$-scheme. Using coherence of the sheaf of cohomologies $\Hosh[-1]\lp\Hhig{\infty}\rp$ we can extend this subsheaf of $\Hhig[-1]{\infty}$ to a locally free coherent $\Bhig[-1]{\infty}\subseteq\Hhig[-1]{\infty}$ such that the cocycles in it project surjectively onto $\Hosh[-1]\lp\Hhig{\infty}\rp$. Proceeding in this way we obtain the required $\Bhig{\infty}\hookrightarrow\Hhig{\infty}$.\end{proof}

Altogether we obtain a dg manifold of finite type $\lp\Grasi{\Hobeg}{\Hoend},\Bhig{\infty}\rp$, that represents a homotopy limit of (\ref{ToCompute}). This is the dg $Quot$-manifold we wanted.

\noindent{\small{\tt{dennis.borisov@uwindsor.ca, lkatzarkov@gmail.com, artan@mit.edu}}

\end{document}